\documentclass{amsart}
\usepackage{epsfig}
\usepackage[english]{babel}
\usepackage{graphics}
\usepackage{amsfonts,dsfont}
\usepackage{amsmath}
\usepackage{latexsym}

\newtheorem{theorem}{Theorem}[section]
\newtheorem{lemma}[theorem]{Lemma}
\newtheorem{proposition}[theorem]{Proposition}
\newtheorem{corollary}[theorem]{Corollary}

\theoremstyle{definition}

\newtheorem{remark}[theorem]{Remark}
\newtheorem{question}[theorem]{Question}

\def\f#1{{\mathcal F}_{#1}}

\def\e{\varepsilon}

\def\R{{\mathbb R}}

\def\N{{\mathbb N}}

\DeclareMathOperator{\Dim}{Dim} 
\numberwithin{equation}{section}

\newcommand{\abs}[1]{\lvert#1\rvert}
\begin{document}

\title[Measures and the Law of the Iterated Logarithm]{Measures and the Law of
  the Iterated Logarithm}  

\author{Imen Bhouri} 
\address{Facult\'e des Sciences de Monastir, 5000 Monastir, Tunisia}
\email{Imen.Bhouri$@$fsm.rnu.tn}

%    General info
\author{Yanick Heurteaux}
\address{Clermont Universit\'e, Universit\'e Blaise Pascal, 
Laboratoire de Math\'ematiques, bp 10448, F-63000 Clermont-Ferrand, France}
\address{CNRS, UMR 6620, Laboratoire de Math\'ematiques, F-63177 Aubi\`ere,  
France}
\email{Yanick.Heurteaux$@$math.univ-bpclermont.fr}

\subjclass[2000]{Primary:28A12,28A78,28D20; Secondary:60F10.}

\keywords{Iterated Logarithm Law, quasi-Bernoulli measure,  
Hausdorff dimension, Packing dimension.}

%%%%%%%%%%%%%%%%%%%%%%%%%%%%%%%%%%%%%%%%%%
 
%%%%%%%%%%%%%%%%%%%%%%%%%%%%%%%%%%%%%%%%

\begin{abstract}Let $m$ be  a unidimensional measure with dimension $d$. A
  natural question is to ask if the measure $m$ is comparable with the
  Hausdorff measure (or the packing measure) in dimension $d$. We give an
  answer (which is in general negative)  to this question in several 
  situations (self-similar measures, 
  quasi-Bernoulli measures). More precisely we obtain fine comparisons between
  the mesure $m$ and generalized Hausdorff type (or packing type)
  measures. The Law of the Iterated Logarithm or estimations of the
  $L^q$-spectrum in a neighborhood of $q=1$ are the tools to obtain such
  results. 
\end{abstract}

\maketitle

\section{Introduction}

For a given probability measure $m$  in $\mathbb R^D$ we define as usual 
\begin{equation}\label{dim}
\begin{cases}
\dim_*(m)=\inf\{\dim(E);~~ m(E)>0\}=\sup\{s\geq 0;~~ m\ll{\mathcal H}^s\}\\
\null\\
\dim^*(m)=\inf\{\dim(E); ~~m(E)=1\}=\inf\{s\geq 0;~~ m\bot{\mathcal H}^s\}
\end{cases}
\end{equation}
respectively the lower and upper dimension of the measure $m$, where 
${\mathcal H}^s$ define the  Hausdorff measure  and $\dim (E)$ is the 
Hausdorff dimension of a set $E$. 

%\noindent
When the equality $\dim_*(m)=\dim^*(m)$ is
satisfied, we say that the measure $m$ is unidimensional and we
denote by $\dim (m)$ the common value. In this situation, the measure $m$ is
carried by a set of dimension $d=\dim (m)$ while $m(E)=0$ for every Borel set
$E$ satisfying $\dim (E)<d$.

For such a unidimensional measure, it is natural to try to compare the measure
$m$ with the Hausdorff measure ${\mathcal H}^d$ and to ask the following
question :

\begin{question}
Does there exist a set $E_0\subset \text{supp}\, (m)$ and a constant $C>0$ 
such that
for every Borel set $A$, 
$m(A)=C\,{\mathcal H}^d(A\cap E_0)$ ?  
\end{question}
Or in a weaker form :
\begin{question}\label{q2}
Does there exist a set $E_0\subset \text{supp}\, (m)$ and a constant $C>0$ 
such that
for every Borel set $A$, 
$\frac1C\,{\mathcal H}^d(A\cap E_0)\le m(A)\le C\,{\mathcal H}^d(A\cap E_0)$ ?  
\end{question}
The answer to Question \ref{q2} is sometimes positive. Let us describe a
classical example. Let $K$ be a self-similar compact set in $\R^D$, that is
\begin{equation}
K=\bigcup_{i=1}^kS_i(K)
\end{equation}
where $S_1,\cdots ,S_k$ are similarities in $\R^D$ with ratio $0<r_i<1$. In the
case where the Open Set Condition is satisfied (see for example \cite{Fa} or
Section \ref{secself} for
a precise definition) it is well known that the Hausdorff dimension of the
compact set $K$ is the unique positive real number $\delta$ such that 
$\sum_{i=1}^kr_i^\delta=1$. A way to obtain this result is the following. Let
$\texttt{\bf p}=(p_1,\cdots ,p_k)$ be 
a probability vector and $m$  be the unique
probability measure such that 
\begin{equation}\label{selfsimilar}
m=\sum_{i=1}^kp_i\,m\circ S_i^{-1}\ .
\end{equation}
The measure $m$ is unidimensional with dimension
\begin{equation}\label{formuledimension}
\dim (m)=\frac{\sum_{i=1}^kp_i\log
  p_i}{\sum_{i=1}^kp_i\log r_i}\ .
\end{equation}
We can refer to \cite{Fal97} or \cite{Ya07} for a proof of this formula. 
The maximal value of $\dim (m)$ in (\ref{formuledimension}) is obtained when
$p_i=r_i^\delta$ for all $i$. In that case, $\dim(m)=\delta$ and it 
is possible to show
that  $m(A)\approx{\mathcal H}^d(A\cap K)$ for all $A$. 
That is the reason why in particular
$\dim (K)=\delta$. As we will see in Section \ref{secself}, this situation is
exceptional. For any other choice of the probability vector
$\texttt{\bf p}$, the measure $m$ is singular with respect to the Hausdorff
measure ${\mathcal H}^{\dim(m)}$. More precisely, using the Law of the
Iterated Logarithm, we will obtain in Theorem \ref{thself} precise
logarithm corrections in the comparison between the measure $m$ and
Hausdorff type measures.

%On the other hand, in numerous works, comparisons between harmonic measure and
%Hausdorff measures are formulated.   
%\noindent

\vskip 0.5cm
Similar quantities involving the packing measure $\widehat{{\mathcal P}}^s$ 
and packing dimension $\Dim$ can be defined
\begin{equation}\label{Dim}
\begin{cases}
\Dim_*(m)=\inf\{\Dim(E);~~ m(E)>0\}
=\sup\{s\geq 0;~~ m\ll\widehat{{\mathcal P}}^s\}
\\
\null\\
\Dim^*(m)=\inf\{\Dim(E); ~~m(E)=1\}
=\inf\{s\geq 0;~~ m\bot\widehat{{\mathcal P}}^s\}\ .
\end{cases}
\end{equation}
They are respectively called the lower and upper packing dimension of the 
measure $m$. 
For more details on packing dimension and packing measures, we can refer to
\cite{Fa} or to the original paper of Tricot \cite{tricot}.

\vskip 0.5cm
There are two fondamental ways to compute or to estimate the dimension of a
measure.

On one hand, the calculation of the dimension of measures may 
be a consequence of some 
independance properties and the use of the strong law of
large numbers. This is in particular the case for self-similar measures under
some separation condition. 

On the other hand, in a more abstract and general context, 
it is well known that the quantities defined in (\ref{dim}) and (\ref{Dim}) 
are strongly related to the
derivatives $\tau'_-(1)$ and $\tau'_+(1)$ 
of the $L^q$-spectrum at point 1 (see \cite{He} or \cite{Ngai}  
for example). In particular, if $\tau'(1)$ exists, the measure is
unidimensional and we have
\begin{equation}\label{equnidim}
-\tau'(1)=\dim_*(m)=\dim^*(m)=\Dim_*(m)=\Dim^*(m)\ .
\end{equation}
This is in particular the case for the so called quasi-Bernoulli measures (see
Section \ref{quasibernoulli} or \cite{BMP} for a precise definition 
and \cite{He} where it
is proved that the $L^q$-spectrum is differentiable in this situation).

\vskip 0.5cm
There are numerous works dealing with the comparison between measures and 
Hausdorff measures or packing measures. 
This can be done in an abstract context 
(\cite{BMP}, \cite{He},  \cite{BBH}, \cite{LN}), in a dynamical context 
(\cite{Lo}, \cite{Y}) or for concrete measures as harmonic measures 
(\cite{Ba}, \cite{Bourgain}, \cite{BH}, \cite{Mak}, \cite{MaV}). 
Logarithmic corrections are also proposed in several situations. In
particular, Makarov and Makarov-Volberg 
obtained such corrections for the harmonic measure, respectively in
Jordan domains (\cite{Mak}) and 
in self-similar Cantor sets (\cite{MaV}). Using 
a dyadic martingale that approximates the logarithm of the densities of the
measure at different scales, Llorente and Nicolau \cite{LLN} 
also obtained some abstract (and in general non explicit) logarithmic 
corrections in the case of doubling measure.

\vskip 0.5cm
In this work we pursue such studies and we try to obtain logarithmic
corrections in several situations, using the following two ideas. On one hand,
the Law of the Iterated Logarithm, which is more precise than 
the Strong Law of Large
Numbers, can be used in the case where some independance
properties are satisfied. On the other hand, 
estimations of $\tau(1+q)-q\tau'(1)$ near $q=0$ where $\tau$ is the
$L^q$-spectrum of the measure $m$ allow us to
derive more precise comparisons beetwen $m$ and generalized Hausdorff
or packing measures. In particular, logarithm or iterated logarithm
corrections may be obtained in some situations.

\vskip 0.5cm
The paper is organised as follows. In Section \ref{bernoulli}, 
we study the case of
Bernoulli products and give fine comparisons with Hausdorff type or packing
type measures. 
In particular such a Bernoulli product $m$ is singular with respect the the 
Hausdorff measure ${\mathcal H}^d$ but absolutely continuous 
with respect to the packing
measure $\widehat{\mathcal P}^d$ where $d$ is the dimension of the measure
$m$. This elementary fact, which is a consequence
of the Law of the Iterated Logarithm seems not to be very present in the
litterature. 
The motivation in studying such a toy example is
that it contains the fundamental ideas that will be developped in the next
sections. 

In Section \ref{secself} we generalize the results of Section \ref{bernoulli}
to the case of self-similar
measures like (\ref{selfsimilar}), when the Open Set Condition is satisfied. 

Finally, the last
section is devoted to the study of quasi-Bernoulli measures. Of
course, in such a general situation we do not have independance properties
which allow to use the Law of the Iterated Logarithm. Nevertheless, in a large
situation we obtain upper 
bounds of type LIL (Section \ref{seclil}) and we prove that fine comparisons
between the measure $m$ and Hausdorff (or packing) type measures in the
reverse sens are strongly
related to the behaviour of the $L^q$-spactrum near the point $q=1$ (Sections
\ref{secreverse} and \ref{secmore}). In particular, in much situations, a
quasi-Bernoulli measure is already singular with respect to the Hausdorff
measure ${\mathcal H}^{\dim(m)}$ but absolutely continuous with respect to the
packing measure $\widehat{\mathcal P}^{\dim(m)}$.
\section{Bernoulli products}\label{bernoulli}
 We begin by the study of a classical example.   Let  
$({\mathcal F}_n)_{n\geq 0}$ be the family of  $\ell$-adic 
cubes of the $n^{\text{th}}$ generation on $[0,1)^D$, in other words :
\begin{equation}\label{fn}
 {\mathcal F}_n=\left\{I=\displaystyle\prod_{i=1}^D\left[k_i/\ell^n, 
(k_i+1)/\ell^n\right); ~~0\leq k_i<\ell^n \right\}.
\end{equation}
Suppose for simplicity  $D=1$ and $\ell=2$. Let $m$ be the Bernoulli 
product on $[0,1)$ with parameter  $0<p<1$. 
 
It is defined as 
follows. If $\varepsilon_1\cdots \varepsilon_n$ are integers
in $\{0,1\}$, and if
$$I_{\varepsilon_1\cdots \varepsilon_n}= \left[
\sum_{i=1}^n\frac{\varepsilon_i}{2^i},
\sum_{i=1}^n\frac{\varepsilon_i}{2^i}+\frac{1}{2^n}\right)\in\f
n$$ then
$$m\left(I_{\varepsilon_1\cdots \varepsilon_n}\right)=p^{s_n}(1-p)^{n-s_n},
\quad\mbox{where}\quad s_n=\varepsilon_1+\cdots +\varepsilon_n\ .$$

It is well known that the measure $m$ is unidimensional with dimension

$$d=-p\log_2p-(1-p)\log_2(1-p)\ .$$ 

This is an easy consequence of the strong law
of large numbers applied to the sequence of independent Bernoulli random
variables $(\e_n)_{n\ge 1}$. Here, the space
$[0,1)$ is equipped with the probability measure $m$ (see for example
\cite{Fal97} or \cite{Ya07}).

It is then natural to think that the Law of the Iterated Logarithm 
gives a more precise result. Curiously this
elementary fact is not very present in the litterature. Let us only mention
\cite{Smo} in which the law of the iterated logarithm is used in a weak form
in order to prove that the measure $m$ is singular with respect 
to the Hausdorff measure ${\mathcal H}^d$.

\begin{proposition}\label{propobernou}
\label{pro1}Let $m$ be a Bernoulli product with parameter
  $0<p<1$ satisfying $p\not= 1/2$. Take
$$d=-p\log_2p-(1-p)\log_2(1-p)\quad\mbox{and}\quad
\sigma^2=p(1-p)\left(\log_2(\frac{p}{1-p})\right)^2$$
and denote  $$\Theta(t)=2^{\sqrt{2\log_2 (1/t)\ln\ln\log_2 (1/t)}}\ .$$
For all $\varepsilon>0$ we have :

\begin{enumerate}
\item $m\ll  {\mathcal H}^\Psi$, where $\Psi(t)=t^d\Theta(t)^{\sigma+\varepsilon}$
\item $m\bot  {\mathcal H}^\Psi$, where 
$\Psi(t)=t^d\Theta(t)^{\sigma-\varepsilon}$. 
\end{enumerate}
In particular, $\dim_*(m)=\dim^*(m)=d$ but $m\bot{\mathcal H}^d$.
\end{proposition}
\begin{remark}
Bernoulli products are particular cases of the 
self-similar measures described in (\ref{selfsimilar}). 
We have $k=2$, $S_1(x)=x/2$, $S_2(x)=(1+x)/2$, and
$\texttt{\bf p}=(1-p,p)$. The self-similar compact set $K$ is the unit interval
$[0,1]$ with dimension $1$.  
The case $p=1/2$ corresponds to the Lebesgue measure :
it is the natural self-similar measure on the compact set $K$. 
In the other cases,
the measure $m$ is singular with respect to 
the Hausdorff measure ${\mathcal H}^{\dim(m)}$. 
\end{remark}
\begin{remark}
In a famous paper (\cite{Mak}), Makorov proved that the harmonic measure of a 
Jordan domain is unidimensional with dimension 1 and obtained similar
iterated logarithm corrections.
\end{remark}
\begin{proof}[Proof of Proposition \ref{propobernou}]
We first prove $(1)$. Let 
$$X_n(x)=-\log_2\left(p^{\e_n}(1-p)^{1-\e_n}\right)=-\log_2
\left(\frac{m(I_n(x))}{m(I_{n-1}(x))}\right)$$
where $I_n(x)$ denotes the unique interval in ${\mathcal F}_n$ containing $x$ 
and $I_n(x)=I_{\varepsilon_1\cdots \varepsilon_n}$. The random variables
$X_n$ are independant and identiquely distributed. An easy
calculation gives  
$${\mathbb E}[X_n]=d\quad \text{and}\quad {\mathbb V}[X_n]=\sigma^2$$ 
where $d$ and $\sigma$ are the quantities introduced in the proposition. 
Derive from the Law of the Iterated Logarithm that, $dm$-almost surely,
\begin{equation}
\label{eq0}
\displaystyle\liminf_{n\rightarrow+\infty} 
\displaystyle\frac{S_n-nd}{\sqrt{2n\ln\ln n}}= -\sigma
\end{equation}
where $S_n=X_1+\cdots  +X_n$. 

Set $\varepsilon>0$. 
Then, $dm$-almost surely,
$$
\exists n_0\in \N^*\ ;\quad\forall n
\geq n_0,\quad S_n\ge nd-(\sigma +\varepsilon)\sqrt{2n\ln\ln n}\ .
$$
Furthermore, $S_n=-\log_2(m(I_n(x)))$. 
Therefore, we get 
$$
\text{a.s,} \hskip 1cm \exists n_0\in \N^*\ ;\quad \forall n
\geq n_0,\quad 
m(I_n(x))\leq \abs{I_n(x)}^d\, \Theta(\abs{I_n(x)})^{\sigma+\varepsilon}\ .
$$
where $\abs{I_n(x)}=2^{-n}$ is the length of the interval $I_n(x)$.
It is classical  to deduce (see for example \cite{Mat}) that for  every set $E$,
$$m(E\cap\liminf_n B_n)\le {\mathcal H}^\Psi(E)$$
where $B_n=\{x\ ;\ m(I_n(x))\leq \Psi(\abs{I_n(x)}) \}$ and  
$\Psi(t)=t^d\Theta(t)^{\sigma+\varepsilon}$. Moreover, the measure $m$ is
carried by the set $\liminf_n B_n$ and the result yields.

We now prove $(2)$. Let $\e>0$. A consequence of (\ref{eq0}) is also that
$dm$-almost surely, 
$$
\label{eq01}
\forall n_0\in\N^*,\ \exists n\ge n_0\ ;\quad S_n\le nd+(-\sigma +\varepsilon)
\sqrt{2n\ln\ln n}
$$
and we get,
$$
\text{a.s},\quad \text{i.o,}\quad 
m(I_n(x))\ge \abs{I_n(x)}^d\, \Theta(\abs{I_n(x)})^{\sigma-\varepsilon}\ .
$$
The full measure set $E_0$ which is just described satisfies ${\mathcal
  H}^\Psi(E_0)<+\infty$ where $\Psi(t)=t^d\Theta(t)^{\sigma-\varepsilon}$. Using
that $\e$ is arbitrary small, we can deduce that $m$ is singular with respect
to ${\mathcal H}^\Psi$ for every $\e>0$. 

In particular, $t^d\ll \Psi(t)$ when $t$ goes to 0, so
that $m\bot {\mathcal H}^d$ 
and the proof of Proposition \ref{pro1} is complete.
\end{proof}

The Law of the Iterated Logarithm also says that $dm$-almost surely,

\begin{equation}
\label{eq02}
\displaystyle\limsup_{n\rightarrow+\infty} 
\displaystyle\frac{S_n-nd}{\sqrt{2n\ln\ln n}}= \sigma\ .
\end{equation}
This asymptotic behavior is deeply related to
comparisons between the measure $m$ and packing measures. That is what is
shown in the following twin proposition.

\begin{proposition}\label{pro2}The notations are the same as in Proposition
  \ref{pro1}.
For all $\varepsilon>0$ we have :

\begin{enumerate}
\item $m\ll  \widehat{{\mathcal P}}^\Psi$, where 
$\Psi(t)=t^d\Theta(t)^{-(\sigma-\varepsilon)}$ 
\item $m\bot  \widehat{{\mathcal P}}^\Psi$, where 
$\Psi(t)=t^d\Theta(t)^{-(\sigma+\varepsilon)}$.
\end{enumerate}
In particular, $\Dim_*(m)=\Dim^*(m)=d$ and  $m\ll\widehat{{\mathcal
    P}}^d$.\\
More precisely, $\widehat{{\mathcal P}}^d(E)<+\infty\Rightarrow m(E)=0$.
\end{proposition}
\begin{proof}
The relation (\ref{eq02}) implies that $dm$-almost surely
$$
\forall n_0\in\N^*,\ \exists n\ge n_0\ ;\quad  S_n\ge nd  
+(\sigma -\varepsilon)\sqrt{2n\ln\ln n}\ .
$$
So,
$$
\text{a.s,}\quad\text{i.o.,}\quad  
m(I_n(x))\le \abs{I_n(x)}^d\Theta(\abs{I_n(x)})^{-(\sigma-\varepsilon)}
$$
and assumption $(1)$ holds with similar arguments as in Proposition \ref{pro1}
(using packing instead of coverings). More precisely, for every set $E$, one
has
$$m(E\cap\limsup_n B_n)\le \widehat{\mathcal P}^\Psi(E)$$
where $B_n=\{x\ ;\ m(I_n(x))\leq \Psi(\abs{I_n(x)}) \}$,  
$\Psi(t)=t^d\Theta(t)^{-(\sigma-\varepsilon)}$ and the sets $B_n$ are such
that $m\left(\limsup_n B_n\right)=1$.

In particular, $t^d\Theta(t)^{-(\sigma-\varepsilon)}\le t^d$ so that
$m\ll\widehat{{\mathcal P}}^d$. More precisely, $t^d
\Theta(t)^{-(\sigma-\varepsilon)}\ll t^d$ when $t$ goes to 0. If $E$ is a set
such that $\widehat{{\mathcal P}}^d(E)<+\infty$, we have successively
$\widehat{{\mathcal P}}^\Psi(E)=0$ and $m(E)=0$.

On the other hand, we have $dm$-almost surely
$$
\exists n_0\in \N\ ;\quad\forall n\geq n_0,\quad S_n\le nd 
+(\sigma +\varepsilon)\sqrt{2n\ln\ln n}
$$
which says that 
$$\text{a.s,} \hskip 1cm \exists n_0\in \N\ ;\quad\forall n
\geq n_0,\quad 
m(I_n(x))\geq \abs{I_n(x)}^d\Theta(\abs{I_n(x)})^{-(\sigma+\varepsilon)}\ .
$$
Let 
$$E_{n_0}=\left\{x\in [0,1)\ ;\ \quad\forall n
\geq n_0,\quad 
m(I_n(x))\geq
\abs{I_n(x)}^d\Theta(\abs{I_n(x)})^{-(\sigma+\varepsilon)}\right\}\ .$$
It is clear that ${\mathcal P}^{\psi_\e}(E_{n_0})\le 1$ where 
$\psi_\e(t)=t^d\theta(t)^{-(\sigma+\e)}$ and 
${\mathcal P}^{\psi_\e}$ is
  the pre-measure related to the packing measure 
$\widehat{\mathcal P}^{\psi_\e}$ (see \cite{Fa} for
  the link between ${\mathcal P}^{\psi_\e}$ and  $\widehat{\mathcal
    P}^{\psi_\e}$). It
  follows that ${\mathcal P}^{\psi_{2\e}}(E_{n_0})=0$ and 
$\widehat{\mathcal P}^{\psi_{2\e}}\left(\bigcup_{n_0}
  E_{n_0}\right)\le\sum_{n_0}{\mathcal P}^{\psi_{2\e}}(E_{n_0})=0$. Moreover,
$m\left(\bigcup_{n_0}E_{n_0}\right)=1$.  
This implies $(2)$ ($\e$ is arbitrary small).
\end{proof}
\section{A more general situation : self-similar measures}\label{secself}
The results established in the previous section are particular cases of the
more general situation of self-similar measures. We can prove the following
general theorem.

\begin{theorem}\label{thself}
Let $K$ be the attractor of a family of similarity transformations
$S_1,\cdots ,S_k$  in $\R^D$ where $S_i$ has similarity  ratio
$0<r_i<1$. Suppose that the Open Set Condition is satisfied and let
$\delta=\dim (K)$ be the Hausdorff dimension of $K$. Recall that $\delta$ is
the unique positive solution of the equation $\sum_{i=1}^k r_i^\delta=1$.

Let $\texttt{\bf p}=(p_1,\cdots ,p_k)$ be 
a probability vector and $m$  be the self-similar 
probability measure such that 
$$
m=\sum_{i=1}^kp_i\,m\circ S_i^{-1}\ .
$$
Set
$$
d=\frac{\sum_{i=1}^kp_i\ln
  p_i}{\sum_{i=1}^kp_i\ln r_i}\quad\text{and}\quad
\sigma^2=\frac{\sum_{i=1}^kp_i(\ln p_i-d\,\ln r_i)^2}{-\sum_{i=1}^kp_i\ln r_i}\ .
$$
Suppose that $d\not=\delta$ (which is
equivalent to $\sigma>0$) and let
$$\theta(t)=e^{\sqrt{2\ln(1/t)\ln\ln\ln(1/t)}}\ .$$
For all $\e>0$ we have :
\begin{enumerate}
\item $m\ll  {\mathcal H}^\Psi$, where $\Psi(t)=t^d\Theta(t)^{\sigma+\varepsilon}$
\item $m\bot  {\mathcal H}^\Psi$, where 
$\Psi(t)=t^d\Theta(t)^{\sigma-\varepsilon}$. 
\end{enumerate}
In particular, $m\bot{\mathcal H}^d$.
\end{theorem}
Recall that the Open Set Condition states that there exists a non-empty and
bounded  open set $U$ in $\R^D$ with  $\bigcup_{i=1}^kS_i(U)\subset U$ and 
$S_i(U)\cap S_j(U)=\emptyset$ for all $i,j$ with $i\not=j$.
\begin{proof}[Proof of Theorem \ref{thself}]
We give the complete proof in the particular case where the strong
separation condition is 
satisfied (i.e. in the case  where the $S_1(K),\cdots ,S_k(K)$ are disjoint
compact sets) and then say a few words in the general case. 

 In the case where the strong separation condition is satisfied, the
 application  

\begin{equation}\label{eqcantor}
\pi\ :\quad i=(i_1,\cdots ,i_n,\cdots )\in\{1,\cdots ,k\}^{\N^*}
\longmapsto\bigcap_nS_{i_1}
\circ\cdots \circ S_{i_n}(K)
\end{equation}
is an homeomorphism between the symbolic Cantor set $\{1,\cdots ,k\}^{\N^*}$
and the self-similar set $K$. Moreover, the measure
 $m$ is nothing else but the image of a multinomial measure on 
$\{1,\cdots ,k\}^{\N^*}$ through this homeomorphism. 
Let
$$K_{i_1\cdots i_n}=S_{i_1}\circ\cdots \circ S_{i_n}(K)\ .$$
For every $x\in K$ there exists a unique sequence
$i_1(x),\cdots ,i_n(x),\cdots $ such that $x\in K_{i_1(x)\cdots i_n(x)}$ for
all $n$. Moreover, the random variables $i_1,\cdots ,i_n,\cdots $ are
independant and uniformly distributed with distribution
$$m(\{i_n=i\})=p_i\quad \forall i\in\{1,\cdots  k\}\ .$$
Set 
$$K_n(x)=K_{i_1(x)\cdots i_n(x)}\quad\text{and}\quad R_n(x)=\abs{K_n(x)}$$
where $\abs{A}$ denotes the diameter of the set $A$.

We may suppose without lost of generality that $\abs{K}=1$ and we define for
every $n\ge 1$ the random variable 
$$S_n(x)=-\ln(m(K_n(x)))+d\,\ln(R_n(x))\quad\text{and}\quad X_n=S_n-S_{n-1}$$
with the convention $S_0=0$.

The random variables $X_n$ are independant, uniformly distributed and take the
value $-\ln p_i+d\,\ln r_i$ with probability $p_i$. An easy calculation gives
$${\mathbb E}[X_n]=0\quad\text{and}\quad{\mathbb V}[X_n]=\sum_{i=1}^kp_i\left(
  \ln p_i-d\,\ln r_i\right)^2=\left(-\sum_{i=1}^kp_i\ln r_i\right)\sigma^2\
.$$ 
The Law of the Iterated Logarithm states that almost surely,
\begin{equation}\label{LIL}
\liminf_{n\to +\infty}\frac{S_n}{\sqrt{2n\ln\ln n}}=-\left(-\sum_{i=1}^kp_i\ln
    r_i\right)^{1/2}\sigma\ .
\end{equation}
On the other hand, $\ln R_n=\rho_1+\cdots +\rho_n$ where the $\rho_j$ 
are independant,
uniformly distributed and such that for all $i\in\{1,\cdots ,k\}$, 
$m(\{\rho_n=\ln r_i\})=p_i$. The strong law of large numbers says that almost
surely, 
\begin{equation}\label{LGN}
\lim_{n\to+\infty}\frac{\ln R_n}n=\sum_{i=1}^kp_i\ln r_i\ .
\end{equation}
Combining (\ref{LIL}) and (\ref{LGN}), we deduce that $dm$-almost surely, 
\begin{equation}\label{liminf}
\liminf_{n\to+\infty}\frac{- \ln
(m(K_n(x)))+d\,\ln(\abs{K_n(x)})}{\sqrt{2\ln\left(\abs{K_n(x)}^{-1}\right)\ln\ln\ln\left(\abs{K_n(x)}^{-1}\right)}}
=-\sigma\ . 
\end{equation}
Let $\e>0$. Using the notation introduced in the theorem, we conclude that
almost surely
$$
\begin{cases}
\exists n_0\ ;\ \forall n\ge n_0,\quad 
m(K_n(x))\le \abs{K_n(x)}^d\theta(\abs{K_n(x)})^{\sigma+\e}\\
\null\\
\forall n_0\ ;\ \exists n\ge n_0,\quad 
m(K_n(x))\ge \abs{K_n(x)}^d\theta(\abs{K_n(x)})^{\sigma-\e}
\end{cases}
$$
The size of the $K_n(x)$ are exponentialy decreasing in the sense that
$$\min_{1\le i\le k}(r_i)\,\abs{K_n(x)}\le \abs{K_{n+1}(x)}\le\max_{1\le i\le
  k}(r_i)
\,\abs{K_n(x)}\ .$$
It is then well known that Hausdorff measures of subsets of $K$ 
computed with coverings using the $K_n(x)$ are comparable to the genuine ones. 
In the same way as in Section
\ref{bernoulli}, we can then conclude that for all $\e>0$,
$$
\begin{cases}
m\ll  {\mathcal H}^\Psi,\quad \text{where}\ 
\Psi(t)=t^d\Theta(t)^{\sigma+\varepsilon}\\
\null\\
m\bot  {\mathcal H}^\Psi,\quad \text{where}\  
\Psi(t)=t^d\Theta(t)^{\sigma-\varepsilon}
\end{cases}
$$
and the proof is finished in the case where the strong separation condition is
satisfied.

In the general case we have to adapt the argument. The difficulty is that the
function $\pi$ defined in (\ref{eqcantor}) is always surjective but not one to
one. We will use the following
lemma which was proved by Graf in \cite{Graf}.

\begin{lemma}\label{OSC}{\em (\cite{Graf})}
The notations are the same as in Theorem \ref{thself}. Under the Open Set
Condition we have :
$$\forall (i_1,\cdots ,i_n)\in\{1,\cdots ,k\}^n,\quad
m(K_{i_1\cdots i_n})=p_{i_1}\cdots  p_{i_n}$$
and
$$\text{if}\quad (i_1,\cdots ,i_n)\not=(j_1,\cdots ,j_n)\quad \text{then}\quad 
m\left(K_{i_1\cdots i_n}\cap K_{j_1\cdots j_n}\right)=0\ .$$
\end{lemma}
We define the following families of subsets of $K$. If $n\ge 1$ and
$(i_1,\cdots ,i_n)\in\{1,\cdots ,k\}^n$, 

$$
K^0_{i_1\cdots i_n}=\left\{x\in K_{i_1\cdots i_n}\ ;\quad \forall
  (j_1,\cdots ,j_n)\not=(i_1,\cdots ,i_n),\quad x\not\in
  K_{j_1\cdots j_n}\right\}\ .
$$ 

It follows from Lemma \ref{OSC} that for every integer  $n\ge 1$ and for every
$(i_1,\cdots ,i_n)\in\{1,\cdots ,k\}^n$, 
$$m(K^0_{i_1\cdots i_n})=p_{i_1}\cdots  p_{i_n}\ .$$
Moreover, the family $K^0_{i_1\cdots i_n}$, where
$(i_1,\cdots ,i_n)\in\{1,\cdots ,k\}^n$, is constituted
of $k^n$ disjoint ${\mathcal G}_\delta$ subsets of $\R^D$ and satisfies :
$$K^0_{i_1\cdots i_nj}\subset K^0_{i_1\cdots i_n}\quad\forall
j\in\{1,\cdots ,k\}\ .$$
Let 
$$K^0=\bigcap_{n\in\N^*}\bigcup_{(i_1,\cdots ,i_n)}
K^0_{i_1\cdots i_n}\ .$$
The set $K^0$ is a ${\mathcal G}_\delta$ subset of $\R^D$ such that
$K^0\subset K$ and $m(K^0)=1$. Moreover for every $x\in K^0$, there exists a
unique sequence $(i_1(x),\cdots ,i_n(x),\cdots )$ such that for every integer
$n\ge 1$,  
$x\in K^0_{i_1(x)\cdots i_n(x)} $. We can extend the applications 
$i_1,\cdots ,i_n,\cdots $ in a
mesurable way and define for every $x\in K$ 
$$K_n(x)=K_{i_1(x)\cdots i_n(x)}$$
such that $x\in K_n(x)$. Moreover the random variables $i_1,\cdots ,i_n,\cdots $
are independant, uniformly distributed and such that $m(\{i_n=i\})=p_i$. We
can already use the Law of the Iterated Logarithm and obtain (\ref{liminf})
which is the key to prove Theorem \ref{thself}.
\end{proof}
\begin{remark}
In the case of Bernoulli products described in Section \ref{bernoulli}, the
sets $K^0_{i_1\cdots i_n}$ are nothing else but the open dyadic intervals of
the $n^{\text{th}}$ generation and $K^0$ is the set of points $x\in [0,1]$
that are not dyadic numbers.
\end{remark} 
\begin{remark}\label{sigma}
It is classical to establish that the $L^q$-spectrum of the measure $m$ 
is given by the implicit
equation 
$$\sum_{i=1}^kp_i^qr_i^{\tau(q)}=1\ .$$
We can refer to \cite{AP} or 
\cite{Fal97} where this formula is obtained and where the link
with multifractal formalism is shown. The function $\tau$ is analytic and an
easy calculation gives $\tau'(1)=-d$ and $\tau''(1)=\sigma^2$. In other words,
$\tau(1-q)=dq+\frac{\sigma^2}2q^2+o(q^2)$ near $q=0$. We will see in Section
\ref{quasibernoulli} that such an estimate is the key to obtain quite similar
results for quasi-Bernoulli measures.
\end{remark}

\vskip 0.5cm
Of course a similar result involving packing measures is also true.
\begin{theorem}\label{thself2}The hypothesis and the notations are the 
same as in Theorem
  \ref{thself}.
For all $\varepsilon>0$ we have :

\begin{enumerate}
\item $m\ll  \widehat{{\mathcal P}}^\Psi$, where 
$\Psi(t)=t^d\Theta(t)^{-(\sigma-\varepsilon)}$ 
\item $m\bot  \widehat{{\mathcal P}}^\Psi$, where 
$\Psi(t)=t^d\Theta(t)^{-(\sigma+\varepsilon)}$.
\end{enumerate}
In particular, $\widehat{{\mathcal P}}^d(E)<+\infty\Rightarrow m(E)=0$.
\end{theorem}
\section{Quasi-Bernoulli measures}\label{quasibernoulli}
Natural generalisations of Bernoulli products or self-similar measures are the
so called quasi-Bernoulli measures.

The notations are the same as in Section \ref{bernoulli}. Suppose that the
$\ell$-adic cubes in ${\mathcal F}_n$ are coded $I_{\e_1\cdots \e_n}$, $0\le
  \e_i<\ell^D$ in such a way that 
$$I_{\e_1\cdots \e_{n+1}}\subset I_{\e_1\cdots \e_n},\quad \quad\forall
\e_1,\cdots ,\e_{n+1}\in\{0,\cdots ,\ell^D-1\}\ .$$
If $I=I_{\e_1\cdots \e_n}\in{\mathcal F}_n$ and
  $J=I_{\e_{n+1}\cdots \e_{n+p}}\in{\mathcal F}_p$, we note $IJ$ the
    $\ell$-adic cube 
$$IJ=I_{\e_1\cdots \e_{n+p}}\in{\mathcal F}_{n+p}$$
      obtained by the concatenation of the words $\e_1\cdots \e_n$ and
      $\e_{n+1}\cdots \e_{n+p}$. 

We say that 
the probability measure $m$ is a quasi-Bernoulli measure on $[0,1)^D$, 
if we can find a constant $C\geq 1$ such that 
\begin{equation}
\forall~~ I, J\in \bigcup_n{\mathcal F}_n,  ~~\frac{1}{C}m(I)m(J)
\leq m(IJ)\leq C m(I)m(J)\ .
\label{eq1}
\end{equation}
Quasi-Bernoulli property appears in many situations. In particular, 
this is the case for the harmonic measure in regular Cantor sets 
(\cite{Ca}, \cite{MaV}) and for the caloric measure in domains delimited 
by Weirstrass type graphs (\cite{BH}).
 The $L^q$-spectrum  $\tau$ is defined as usual by 
$$
\tau(q)=\limsup_{n\to +\infty}\tau_n(q)\quad\text{with}\quad\tau_n(q)=
\displaystyle\frac{1}{n\log \ell}
\log\left(\sum_{I\in {\mathcal F}_n} m(I)^q\right)\ .
$$

In the case of quasi-Bernoulli measures, sub and super multiplicative 
properties of the sequences
$$C^{\abs{q}}\sum_{I\in {\mathcal F}_n} m(I)^q\quad\text{and}\quad
C^{-\abs{q}}\sum_{I\in {\mathcal F}_n} m(I)^q $$
ensure that the sequence $\tau_n(q)$ converges and satisfies
\begin{equation}\label{eq4.2}
C^{-\abs{q}}\ell^{n\tau(q)}\le \sum_{I\in {\mathcal F}_n} m(I)^q\le 
C^{\abs{q}}\ell^{n\tau(q)}\ .
\end{equation}
We can see \cite{BMP}, \cite{He} or \cite{Ya07} for more details.

It is well known that quasi Bernoulli measures satisfy the multifractal
formalism (see \cite{BMP}) and it is proved in \cite{He} that the
$L^q$-spectrum  is of class $C^1$ on $\R$. In particular, according to
(\ref{equnidim}),  quasi-Bernoulli
measures are unidimensional measures with dimension 
$$d=-\tau'(1)\ .$$

The $L^q$-spectrum $\tau$ and the dimension $d$ of the measure $m$ have the
following probabilistic interpretations which
are detailed in \cite{Ya07}. If $I_n(x)$ is the unique cube in 
${\mathcal F}_n$ containing $x$, 
let
$$
\frac{S_n}{n}=\frac{X_1+\cdots  +X_n}{n}\hskip 1cm\text{and}\hskip 1cm X_n(x)=
-\log_\ell\left(\frac{m(I_n(x))}{m(I_{n-1}(x))}\right)\ .
$$
In other words,
$$\frac{S_n}{n}=
\frac{\log\left(m(I_n(x))\right)}{\log\left(\abs{I_n(x)}\right)}$$
where $\abs{I_n(x)}=\ell^{-n}$ is the ``length'' of the cube $I_n(x)$. The
asymptotic behavior of the sequence of random variables $S_n/n$ is then deeply
related to the local behavior of the measure $m$ and the dimension $d$ of the
measure $m$ is the almost sure limit of the sequence of random variables
$S_n/n$. Moreover, 
$$\tau_n(1-q)=\frac1n\log_\ell{\mathbb
  E}[\ell^{qS_n}]\quad\text{and}\quad\tau(1-q)= 
\lim_{n\to +\infty}\frac1n\log_\ell{\mathbb
  E}[\ell^{qS_n}]$$
are related to the log-Laplace transform of the sequence $S_n$. Finally,
(\ref{eq4.2}) can be rewritten
\begin{equation}\label{eq4.3}
C^{-1}\,\ell^{n\tau(1-q)}\le {\mathbb E}[\ell^{qS_n}]\le
C\,\ell^{n\tau(1-q)} 
\end{equation}
where the constant $C$ is independant of $n$ and independant of $q$, provided
$q$ stays in a bounded set. Inequalities (\ref{eq4.3}) will be usefull in the
following sections.

There exists a symbolic counterpart $\mu$ to the quasi-bernoulli 
measure $m$ which is
defined on the symbolic Cantor space $\{0,\cdots,\ell^D-1\}^{\N^*}$ 
as the image of
$m$ through the application 
$$
J(x)=(\e_i)_{i\ge 1}\quad\text{if}\quad\{x\}=\bigcap_{n\ge
  1}I_{\e_1\cdots\e_n}\ .
$$
Carleson observed in \cite{Ca} that such a quasi-Bernoulli measure $\mu$ on
the Cantor set $\{0,\cdots,\ell^D-1\}^{\N^*}$ is strongly equivalent to a
mesure $\tilde\mu$ (that is $\frac1C \mu\le \tilde \mu\le C\mu$ 
for some constant
$C\ge 1$) which is shift-invariant and ergodic, where the shift operator $S$
is defined by
$$
S\ :\ (\e_i)_{i\ge 1}\in \{0,\cdots,\ell^D-1\}^{\N^*}\longmapsto 
(\e_i)_{i\ge 2}\in \{0,\cdots,\ell^D-1\}^{\N^*}\ .$$

Coming back to $m$, it follows that $m$ is strongly equivalent to a
quasi-Bernoulli measure $\tilde m$ which is $T$-invariant and ergodic where
$T$ is the ``shift'' operator on $[0,1)^D$ defined by
\begin{equation}\label{eqT}
T\ :\ x=\bigcap_{n\ge 1}I_{\e_1\cdots\e_n}\longmapsto 
Tx=\bigcap_{n\ge 2}I_{\e_2\cdots\e_n}\ .
\end{equation}
This will be a key in Section \ref{secreverse} and \ref{secmore}.

Let us finally describe the closed support of the quasi-Bernoulli measure
$m$. If ${\mathcal G}_n$ is the set of $\ell$-adic cubes $I\in {\mathcal F}_n$
such that $m(I)>0$, it is clear that 
$$\text{supp}\,(m)=\bigcap_{n\ge 1}\bigcup_{I\in{\mathcal G}_n}\bar I\ .
$$
More precisely, let 
$$G=\left\{\e\in\{0,\cdots,\ell^D-1\}\ ;\quad m(I_\e)>0\right\}\ .$$
Quasi-Bernoulli property ensures that
$${\mathcal G}_n=\left\{I_{\e_1\cdots\e_n}\ ;\quad\forall i\in\{1,\cdots,n\},\ 
\e_i\in G\right\}\ .$$
In other words, in the symbolic counterpart, the associated measure $\mu$ is
constructed on the smaller Cantor set $G^{\N^*}$.

Let $g=\sharp(G)$ be the cardinal of the set $G$. Define the homogeneous
probability  measure
$m_0$ on $\text{supp}\,(m)$ by the formula :
$$m_0(I)=g^{-n},\quad\forall I\in {\mathcal G}_n\ .$$ 

Elementary properties of the measure $m_0$ allow us to conclude that the
dimension $\delta$ of the compact set $\text{supp}\, (m)$ satisfies 
$\delta=\log_\ell g$ and that there exists a constant $C>0$ such that for every
set Borel set  $A$,
$$\frac1C\,{\mathcal H}^\delta(A\cap \text{supp}\,(m))\le m_0(A)
\le C\,{\mathcal H}^\delta(A\cap \text{supp}\,(m))\ .$$

\subsection {A bound of type LIL}\label{seclil}

According to Remark \ref{sigma}, it is natural to think that the quadratic
term in the development of $\tau(1-q)$ near $q=0$ gives logarithmic
corrections in the comparison between $m$ and Hausdorff types measures. We are
able to establish such estimations in the case of quasi-Bernoulli measures. 

\begin{theorem}\label{th3} Let $m$ be a quasi-Bernoulli measure with dimension
  $d=-\tau'(1)$. Suppose that 
there exists a real $\sigma\ge 0$ such that $\tau(1-q)=qd+\displaystyle
\frac{\sigma^2}{2}q^2+o(q^2)$ in a neighborhood of $0$.  
Then, $dm$-almost surely,
\hskip3cm
$$
\begin{cases}
\displaystyle\limsup_{n\rightarrow+\infty} \displaystyle\frac{S_n-nd}
{\sqrt{2n\log_\ell\log_\ell n}}\leq \sigma\\ 
\null\\
\displaystyle\liminf _{n\rightarrow+\infty}\displaystyle\frac{S_n-nd}
{\sqrt{2n\log_\ell\log_\ell n}}\geq -\sigma \ . 
\end{cases}
$$
\end{theorem}

\begin{remark}
Theorem \ref{th3} remains true when $\sigma=0$. In that case, the conclusion
is $\displaystyle\lim_{n\to +\infty}\frac{S_n-nd}
{\sqrt{n\log_\ell\log_\ell n}}=0$ $dm$-almost surely.
\end{remark}

\begin{remark}
In general, we do not know if $\tau''(1)$ exists. Nevertheless, an important
class of quasi-Bernoulli measures is constituted of Gibbs measures associated
to an H\"older potential. 
In such a case, the $L^q$-spectrum is known to be analytic 
(see for example \cite{Zin}) and the hypothesis in Theorem \ref{th3} are
satisfied.
\end{remark}
\begin{remark}Return to the case of Bernoulli products described in Section
  \ref{bernoulli}. An easy calculation
  gives 
$$\tau(q)=\log_2\left(p^q+(1-p)^p\right)\quad\text{and}\quad
  \tau''(1)=\frac{p(1-p)}{\ln 2}\left(\ln(\frac{p}{1-p})\right)^2
=(\ln 2)\,{\mathbb V}[X_n]\ .$$ 
Here, $\ell=2$ and the coefficient $\ln 2$ is due to the fact
that functions  $\Theta$
are not similarly normalised in Section \ref{bernoulli} and in  Theorem
\ref{th3}. 
\end{remark}

In order to prove Theorem \ref{th3}, we  need the following lemma, which is
some kind of maximal lemma adapted to the situation . 
\begin{lemma}\label{lem1}
Let $\varepsilon>0$, $a>0$ and $n_0<n_1$ be two integers. Then, for 
$\displaystyle\frac{a}{n_1(\sigma+\varepsilon)^2}$ small enough, one has
\begin{equation}
\label{eq2}
\displaystyle m\left\{ \sup_{k\in \{
n_0,\cdots ,n_1\}}(S_k-kd)\geq a\right\}\leq \displaystyle C\, 
\ell^{\frac{-a^2}{2n_1(\sigma+\varepsilon)^2}}
\end{equation}
where $C$ is a constant independent of all parameters. 
\end{lemma}
\begin{proof}
For $n_0\le k\le n_1$, let 
$$A_k=\left\{ x\ ;\quad (S_j-jd<a\ \ \text{if}\  n_0\leq
  j<k)\quad\text{and}\quad S_k-kd\geq a\right\}.$$
We have to estimate $m\left(\bigcup_{k=n_0}^{n_1}A_k\right)$. Observe that
$A_k$ is the union of some cubes in ${\mathcal F}_{k}$ and denote by 
$${\mathcal A}_k=\left\{I\in {\mathcal F}_{k}\ ;\ I\subset A_k\right\}\ .$$
Let $0<q<1$. According to (\ref{eq4.3}), we have
\begin{eqnarray*}
\mathbb E [\ell ^{qS_{n_1}}\mathds1_{A_k}]&=& 
\sum_{K\in{\mathcal F}_{n_1}\ K\subset A_k}m(K)^{1-q}\\
&=&\sum_{I\in {\mathcal A}_k, J\in {\mathcal F}_{n_1-k}}m(IJ)^{1-q}\\
&\ge& C\sum_{I\in {\mathcal A}_k}m(I)^{1-q}\sum_{J\in {\mathcal F}_{n_1-k}}
m(J)^{1-q}\\
&=& C\sum_{I\in {\mathcal A}_k}m(I)^{1-q}\mathbb E [\ell^{qS_{n_1-k}}]\\
&\ge& C\,\mathbb E [\ell^{qS_{k}}\mathds1_{{\mathcal A}_k}]\ell^{(n_1-k)\tau(1-q)}\\
&\ge& C\,m(A_k)\ell^{q(kd+a)}\ell^{(n_1-k)\tau(1-q)}\ .
\end{eqnarray*}
The constant $C$ can change from line to line but is independant of
$k$, $n_1$ and $q\in [0,1]$. 

Remember that $\tau$ is convex and $d=-\tau'(1)$. We can finally find a constant $C>0$, such that 
\begin{eqnarray*}
m(A_k)&\leq& C\,\mathbb E [\ell ^{qS_{n_1}}\mathds1_{A_k}]
\ell^{-q(kd+a)}\ell^{-(n_1-k)\tau(1-q)}\\
&\leq& C\,\mathbb E [\ell ^{qS_{n_1}}\mathds1_{A_k}]
\ell^{-q(kd+a)}\ell^{-(n_1-k)dq}\\
&=& C\,\mathbb E [\ell ^{qS_{n_1}}\mathds1_{A_k}]\ell^{-qa-qn_1d}\ .
\end{eqnarray*}
Let $\e>0$. If $q$ is small enough, we get
\begin{eqnarray*}
m(\bigcup_{k=n_0}^{n_1}A_k)&\leq& C\,\mathbb E [\ell ^{qS_{n_1}}]\ell^{-qa-qn_1d}\\
&\leq& C\,\displaystyle \ell^{n_1(\tau(1-q)-qd)}\ell^{-qa}\\
&\leq& C\,\displaystyle\ell^{n_1(\sigma+\varepsilon)^2\frac{q^2}{2}-qa}\ .
\end{eqnarray*}
This estimate is optimal for 
$q=\frac{a}{n_1(\sigma+\varepsilon)^2}$. Finally, if
$\frac{a}{n_1(\sigma+\varepsilon)^2}$ is small enough, we obtain 
$$m(\bigcup_{k=n_0}^{n_1}A_k)\leq
C\,\ell^{-\frac{a^2}{2n_1(\sigma+\varepsilon)^2}}$$
which concludes the proof of Lemma \ref{lem1}.
\end{proof}
\noindent
\begin{proof}[Proof of Theorem \ref{th3}]
The proof of Theorem \ref{th3} is quite standard. Fix $\e>0$ and choose
$\alpha>1$ such that 
$\frac{(\sigma+2\varepsilon)^2}{\alpha(\sigma+\varepsilon)^2}>1$. For $k\in
\mathbb N$, let 
$n_k=[\alpha^k]$, (the integrand part of $\alpha^k$) and
$$
B_k=\left\{\exists\ n\in \{n_k,\cdots ,n_{k+1}\} ~~:~~S_n -nd\geq 
(\sigma+2\varepsilon)\sqrt{2n_k\log_\ell\log_\ell n_k} \right\}
$$
An easy calculation proves that 
$$\frac{(\sigma+2\varepsilon)\sqrt{2n_k\log_\ell\log_\ell
    n_k}}{n_{k+1}(\sigma+\e)^2}$$
goes to 0 when $k\to +\infty$, so that we can apply Lemma \ref{lem1} when $k$
is large enough. We get
$$m(B_k)\le C\,\left[\log_\ell
  n_k\right]^{-\frac{(\sigma+2\e)^2n_k}{(\sigma+\e)^2n_{k+1}}}\ .$$
We claim that
$$\left[\log_\ell
  n_k\right]^{-\frac{(\sigma+2\e)^2n_k}{(\sigma+\e)^2n_{k+1}}}\sim
\left[k\log_\ell\alpha\right]^{-\frac{(\sigma+2\e)^2}{\alpha(\sigma+\e)^2}}$$
so that $\displaystyle\sum_km(B_k)$ converges. Hence, by the 
Borel-Cantelli lemma, almost surely, 
only finitely many of these events occur, which 
achieves the proof of  the first estimation of Theorem \ref{th3}. The
second part of Theorem \ref{th3} can be proved in the same way.
\end{proof}

Theorem \ref{th3} allows us to compare the measure $m$ with Hausdorff and
packing measures . That is what is done in the following corollary.

 \begin{corollary} Let $\Theta(t)=
\ell^{\sqrt{2\log_\ell1/t\log_\ell\log_\ell\log_\ell1/t}}$. Then, for all 
$\varepsilon>0$,
\begin{enumerate}
\item $m\ll  {\mathcal H}^\Psi$, where $\Psi(t)=t^d\Theta(t)^{\sigma+\varepsilon}$
\item $m\bot  \widehat{\mathcal P}^\Psi$, where 
$\Psi(t)=t^d\Theta(t)^{-(\sigma+\varepsilon)}$.
\end{enumerate} 
\end{corollary}
\begin{proof}
The arguments are the same as in Section \ref{bernoulli}.
\end{proof}
\subsection{Estimations in the reverse sense}\label{secreverse}
The lack of independance does not allow us to have so precise estimations
in the
reverse sens. 
Nevertheless, we have the following general result.
\begin{theorem}
\label{th4}
Let $m$ be a quasi-Bernoulli measure with dimension $d=-\tau'(1)$. 
Suppose that there exists a real 
$\sigma>0$ such that 
$\tau(1-q)=qd+\displaystyle\frac{\sigma^2}{2}q^2+o(q^2)$ in a neighborhood 
of $0$. Then, almost surely,
\hskip3cm
$$
\begin{cases}
\displaystyle\limsup_{n\rightarrow+\infty}
\displaystyle\frac{S_n-nd}{\sqrt{n}}= 
+\infty\\
\null\\ 
%\text{and}\\
\displaystyle\liminf_{n\rightarrow+\infty}
\displaystyle\frac{S_n-nd}{\sqrt{n}}= 
-\infty
\end{cases}
$$
\end{theorem}
\begin{proof}  Remember that the mesure $m$ is strongly
  equivalent to a 
  quasi-Bernoulli measure $\tilde m$ which is ``shift'' invariant and ergodic
  (see the introduction of the section). The
  $L^q$-spectrum is the same for the two measures. Moreover, with obvious
  notations, there exists a constant $C>0$ independant of $n$ and $x$ such
  that $\abs{S_n-\tilde S_n}\le C$. It follows that the 
  asymptotic behavior of the quantity
  $\frac{S_n-nd}{\sqrt{n}}$ is the same for the measure $m$
  and the  measure
  $\tilde m$. We can then
  assume, without lost of generality, that the measure $m$ is ``shift'' 
  invariant and ergodic.

Let $A>0$ and $0<q<1$, 
\begin{eqnarray*}
\mathbb E [\ell ^{qS_n}]&=& \mathbb E [\ell
^{qS_n}\mathds1_{S_n\le nd+A\sqrt{n}}]
+\mathbb E [\ell ^{qS_n}\mathds1_{S_n> nd+A\sqrt{n}}]\\
&\leq& \ell ^{q(nd+A\sqrt{n})}+m(\{S_n> nd+A\sqrt{n}\})^{1/2}
\mathbb E [\ell ^{2qS_n}]^{1/2}\ .
\end{eqnarray*}
According to (\ref{eq4.3}), 
we have
$$
c_1 \ell^{n\tau(1-q)}\leq 
\ell
^{q(nd+A\sqrt{n})}+c_2m(\{S_n>nd+A\sqrt{n}\})^{1/2}\ell^{(n/2)\tau(1-2q)}\ .
$$
So, if $\e>0$ and $q$ is small enough,
\begin{eqnarray*}
m(\{S_n> nd+A\sqrt{n}\})^{1/2} &\geq& 
\displaystyle\frac{c_1\ell^{n(\tau(1-q)-qd)-qA\sqrt{n}}-1}
{c_2\ell^{n(\frac{1}{2}\tau(1-2q)-qd)-qA\sqrt{n}}}\\
&\geq& \displaystyle\frac{c_1\ell^{n(\sigma-\varepsilon)^2\frac{q^2}{2}-
qA\sqrt{n}}-1}{c_2\ell^{n(\sigma+\varepsilon)^2q^2-qA\sqrt{n}}}\ .
\end{eqnarray*}
Take $q=\frac{\lambda}{\sqrt{n}}$. We get
$$
m(\{S_n> nd+A\sqrt{n}\})^{1/2}\ge 
\frac{c_1\ell^{(\sigma-\varepsilon)^2\frac{\lambda^2}{2}-
\lambda A}-1}{c_2\ell^{(\sigma+\varepsilon)^2\lambda^2-\lambda A}}\ .
$$
We can choose $\lambda$ large enough such that $c_1\ell^{(\sigma-\varepsilon)^2
\frac{\lambda^2}{2}-\lambda A}-1 >0$. It follows that there exists a constant
$c>0$ such that for sufficiently large $n$, 
$$m(\left\{S_n> nd+A\sqrt{n}\right\})\ge c\ .$$
Finally
$$ m(\left\{S_n> nd+A\sqrt{n} ~~i.o\right\})\ge c\ .$$
Recall that $S_n=-\log_\ell (m(I_n(x)))$. Quasi-Bernoulli property implies
that for $dm$-almost all $x\in [0,1)^D$, 
$$\abs{\log_\ell(m(I_n(x)))-\log_\ell(m(I_n(Tx)))}\le C,$$
where $T$ is the ``shift'' operator described in (\ref{eqT}).  
Finally, the set
$$\left\{S_n> nd+A\sqrt{n} ~~i.o\right\}$$ 
is shift invariant and we can
conclude that 
$$ m(\left\{S_n> nd+A\sqrt{n}~~i.o\right\})=1\ .$$ 
The real $A$ being arbitrary large, we obtain the first part of Theorem
\ref{th4}. We can prove the second part of Theorem \ref{th4} is a similar way, 
using estimations of $\mathbb E [\ell ^{qS_n}]$ with $q<0$.
\end{proof}

\begin{corollary}
The hypothesis are the same as in Theorem \ref{th4}. 
Let $a\in\R$ and $\psi_a(t)=t^d\ell^{a\sqrt{\log_\ell 1/t}}$. We have 
 $$\forall a>0,\ m\bot  {\mathcal H}^{\psi_a}\hskip 1cm \text{and} \hskip 1cm
 \forall a<0,\ m\ll \widehat{{\mathcal P}}^{\psi_a}\ .$$
\end{corollary}
\begin{proof}
Let $a\in\R$. Theorem \ref{th4} naturally implies that $dm$-almost surely,
infinitely often, 
$$
\begin{cases}
m(I_n(x))\le\abs{I_n(x)}^d\ell^{a\sqrt{-\log_\ell(\abs{I_n(x)})}}\\
\null\\
m(I_n(x))\ge\abs{I_n(x)}^d\ell^{a\sqrt{-\log_\ell(\abs{I_n(x)})}}
\end{cases}
$$
which gives the conclusion with similar arguments as in Section
\ref{bernoulli}. 
\end{proof}
In particular we can deduce :
\begin{corollary}
The measure $m$ satisfies the following properties :
\begin{enumerate}
\item There exists $E\subset \R^D$ such that $m(E)=1$ and 
${\mathcal H}^d(E)=0$. In particular $m\bot  {\mathcal H}^d$.
\item If $\widehat{{\mathcal P}}^d(E)<+\infty$ then $m(E)=0$. In particular 
$m\ll \widehat{{\mathcal P}}^d$.
\end{enumerate}
\end{corollary}
\subsection{More general estimations}\label{secmore}
As remarked in the previous section, in the general case, we do not know 
if $\tau''(1)$
exists and is strictly positive. Nevertheless, in the general case we can
obtain the less precise following result.

\begin{theorem}\label{th+}
Let $m$ be a quasi-Bernoulli measure with dimension $d=-\tau'(1)$. Let
$\chi(q)=\tau(1-q)-qd$. Suppose that there exists a constant $C>0$ such that
\begin{equation}\label{notflat}
\forall q\in(0,1],\quad 0<\chi(q)\le C\,\chi(q/2)\ .
\end{equation}
Then, $dm$-almost surely,
$$\limsup_{n\to
  +\infty}\frac{S_n-nd}{\theta(n)}\ge 1,$$ 
where $\theta(t)=\frac{1}{\chi^{-1}(1/t)}$ and $\chi^{-1}$ is the inverse
fonction of $\chi$ on $[0,\chi(1)]$.\\
As a consequence, for all $0<a<1$ we have
$$m\ll\widehat{\mathcal P}^{\psi_a}\quad\text{where}
\quad\psi_a(t)=t^d\ell^{-a\theta(\log_\ell 1/t)}\
.$$
In particular, $m$ is absolutely continuous with
respect to $\widehat{\mathcal P}^d$. 
\end{theorem}
\begin{remark}\label{remarkbeta}
Hypothesis (\ref{notflat}) states that function $\chi$ 
is not flat when $q\to 0$,
$q>0$. We know that $\chi$ is a continuous convex function such that 
$\chi(0)=0$.
It follows  that $\chi(q)\ge 2\,\chi(q/2)$. It is then easy to check that
under hypothesis (\ref{notflat}), 
there exists $\alpha> 1$ and $C>0$ such that when
$q>0$ is small enough 
$$\chi(q)\ge C\, q^\alpha\ .$$
Whe can then deduce that 
$$
\theta(t)\le C\,t^{1/\alpha}
$$
for some $C>0$ and we can replace $\theta(t)$ by $t^{\beta}$ 
with $\beta=1/\alpha$ in  Theorem \ref{th+}.
\end{remark}
\vskip 0.5cm
Of course, we can obtain a similar type result if we have some information on
$\chi(q)$ when $q\to 0$ and $q<0$. 
\begin{theorem}\label{th-}
The notations are the same as in Theorem \ref{th+}. Suppose that there exists
a constant $C>0$ such that 
\begin{equation}
\forall q\in[-1,0),\quad 0<\chi(q)\le C\,\chi(q/2)\ .
\end{equation}
Then, $dm$-almost surely,
$$\liminf_{n\to
  +\infty}\frac{S_n-nd}{\theta(n)}\le-1,$$ 
where $\theta(t)=\frac{1}{\chi^{-1}(-1/t)}$ and $\chi^{-1}$ is the inverse
fonction of $\chi$ on $[0,\chi(-1)]$.\\
As a consequence, we have, for all $0<a<1$,
$$m\bot{\mathcal H}^{\psi_a}\quad\text{where}
\quad\psi_a(t)=t^d\ell^{a\theta(\log_\ell 1/t)}\ .$$
In particular, $m$ is singular with
respect to ${\mathcal H}^d$. 
\end{theorem}
\begin{remark}
As observed in Remark \ref{remarkbeta} we can replace $\theta(t)$
by $t^\beta$ for some $\beta<1$ in the conclusions of Theorem \ref{th-}.
\end{remark}

We now give the proof of Theorem \ref{th+}. The proof of Theorem \ref{th-} is
similar.

\begin{proof}[Proof of Theorem \ref{th+}.]The function $\chi$ is a continuous
convex function on $[0,1]$ such that $\chi(0)=0$  and $\chi(q)>0$ if
$q>0$. It follows that $\chi$ is increasing and we can define the 
inverse $\chi^{-1}$ on $[0,\chi(1)]$.

Let $0<a<1$ and $q>0$ sufficiently small. 
Using the same argument as in Theorem \ref{th4}, we have 
\begin{eqnarray*}
m(\{S_n\ge nd+a\theta(n)\})^{1/2}&\geq& 
\displaystyle\frac{c_1\ell^{n\chi(q)-aq\theta(n)}-1}
{c_2\ell^{\frac{n}{2}\chi(2q)-aq\theta(n)}}\\
&\geq& \displaystyle\frac{c_1\ell^{n\chi(q)-aq\theta(n)}-1}
{c_2\ell^{\frac{nC}{2}\chi(q)-aq\theta(n)}}\ .
\end{eqnarray*}
If $\lambda>0$ and $q=\chi^{-1}(\lambda/n)$, we get
$$m(\{S_n\ge nd+a\theta(n)\})^{1/2}\ge
\frac{c_1\ell^{\lambda-a\chi^{-1}(\lambda/n)\theta(n)}-1} 
{c_2\ell^{\frac{C\lambda}{2}-a\chi^{-1}(\lambda/n)\theta(n)}}\ .
$$
Recall that $\chi^{-1}$ is a concave function on $[0,\chi(1)]$ such that
$\chi^{-1}(0)=0$. It follows that $\chi^{-1}(\lambda t)\le\lambda\chi^{-1}(t)$
if $\lambda\ge 1$ and $0\le \lambda t\le \chi(1)$. Finally,
$$\lambda-a\chi^{-1}(\lambda/n)\theta(n)\ge\lambda(1-a)$$
if $\lambda\ge 1$ and $n\ge \lambda/\chi(1)$.

Choose $\lambda$ such that $c_1\ell^{\lambda(1-a)}-1>0$. We get
$$m(\{S_n\ge nd+a\theta(n)\})^{1/2}\ge
\frac{c_1\ell^{\lambda(1-a)}-1} 
{c_2\ell^{\frac{C\lambda}{2}-a\chi^{-1}(\lambda/n)\theta(n)}}\ge
\frac{c_1\ell^{\lambda(1-a)}-1} 
{c_2\ell^{\frac{C\lambda}{2}}}=c>0
$$
if $n$ is sufficiently large. The end of the proof is the same as in Theorem
\ref{th4}. 
\end{proof}
Theorem \ref{th+} and Theorem \ref{th-} can be applied in the important case
where the function $\tau$ is analytic. This is in particular the case when the
measure $m$ is a Gibbs measure associated to an H\"older potential (see
\cite{Zin}). 

\begin{corollary}
Let $m$ be a quasi-Bernoulli measure in $[0,1)^D$. 
Let 
$$\delta=\dim\,(\text{\em supp}\,(m))\quad \text{and}\quad d=\dim(m)\ .$$
Suppose that  $\tau$ is analytic. There are only two possible cases : 
\begin{enumerate}
\item[(i)] $d=\delta$ and the measure $m$ is strongly equivalent 
to the Hausdorff
  measure ${\mathcal H}^\delta$ on $\text{\em supp}\,(m)$.  
\end{enumerate}
\noindent or
\begin{enumerate}
\item[(ii)] $d<\delta$ of the measure $m$ is
  singular with respect to ${\mathcal H}^d$ but absolutely continuous with
  respect to $\widehat{\mathcal P}^d$.
\end{enumerate}
\end{corollary}

\begin{proof}
As in the introduction of the section, denote by $m_0$ the homogeneous
measure on $\text{supp}\,(m)$. The measure $m_0$ is strongly equivalent to the
Hausdorff measure ${\mathcal H}^\delta$ on $\text{supp}\,(m)$. 
Using a similar argument as in \cite{Ya07} 
Corollary 5.5, it is classical to prove that, if the 
quasi-Bernoulli measure $m$ is not strongly
equivalent to the measure $m_0$, its dimension $d$ satisfies 
$d<\delta$. On the other hand, we know that
$\tau(0)=\dim(\text{supp}(m))=\delta$. In the case where $d<\delta$, we can
then conclude that $\tau(1-q)\not\equiv dq$. 
If $\tau$ is analytic, we obtain that there exists a smallest integer 
$n\ge 2$ such that $\tau^{(n)}(1)\not=0$. Moreover $\tau$ is convex. It
follows that $n=2k$ is even and $\tau^{(n)}(1)=\lambda$ is a strictly positive
real number. We can then write 
$$\tau(1-q)=dq+\lambda q^{2k}+o(q^{2k})$$
in a neighborhood  of $q=0$. Finally, the hypothesis of Therorem \ref{th+} and
Theorem \ref{th-} are satisfied.
\end{proof}

\end{document}